\documentclass[12pt]{amsart}
\usepackage{amssymb,mathrsfs,color,bbold}
 \usepackage[all]{xy}
\usepackage{color}

\theoremstyle{plain}
\newtheorem{theorem}{Theorem}[section]
\newtheorem{proposition}[theorem]{Proposition}
\newtheorem{corollary}[theorem]{Corollary}
\newtheorem{lemma}[theorem]{Lemma}

\theoremstyle{definition}

\newtheorem{example}[theorem]{Example}
\newtheorem{definition}[theorem]{Definition}

\newcommand{\abs}[1]{\lvert#1\rvert}
\newcommand{\norm}[1]{\lVert#1\rVert}

\newcommand{\la}{\langle}
\newcommand{\ra}{\rangle}

\renewcommand{\span}{\operatorname{span}}

\newcommand{\R}{{\mathbb R}}

\newcommand{\Xn}{{X_n^\sim}}
\newcommand{\Xnn}{{(X_n^\sim)_n^\sim}}

\newcommand{\Xu}{{X_{uo}^\sim}}
\newcommand{\al}{\alpha}

\author[N.~Gao]{Niushan Gao}
\address{Department of Mathematics and Computer Science, 
University of Lethbridge, Lethbridge, Canada T1K 3M4}
\email{gao.niushan@uleth.ca}

\author[D.~Leung]{Denny H.~Leung}
\address{Department of Mathematics, National University of Singapore, Singapore
117543}
\email{matlhh@nus.edu.sg}

\author[F.~Xanthos]{Foivos Xanthos}
\address{Department of Mathematics, 
Ryerson University, 350 Victoria St.,
Toronto, ON, M5B 2K3, Canada.}
\email{foivos@ryerson.ca}

\title[Duality for uo-convergence]{Duality for unbounded order convergence and
applications}



\thanks{The first author is a PIMS Postdoctoral Fellow. The second author is
partially supported by AcRF grant R-146-000-242-114. The third author
acknowledges the support of an NSERC grant.}

\date{\today}

\begin{document}

\begin{abstract}
Unbounded order convergence has lately been systematically
studied as a generalization of almost everywhere convergence to the abstract
setting of vector and Banach lattices. This paper presents a duality theory for
unbounded
order convergence. We define the unbounded order dual  (or uo-dual) $\Xu$ of a
Banach lattice $X$  and identify it as the order continuous
part of the order continuous dual $\Xn$.  The result allows us
to characterize the Banach lattices that have order continuous preduals and to
show that an order continuous predual is unique when it exists.  
Applications to the Fenchel-Moreau duality theory of convex functionals are
given.  
The applications are of interest in the theory of risk measures in Mathematical
Finance.
\end{abstract}

\maketitle

\section{Introduction}

Let $X$ be a vector lattice.  A net $(x_\alpha)$ in $X$ is said to {\em order
converge} to $x\in X$, written as $x_\al\stackrel{o}{\to}x$,  if there is
another net $(y_\gamma)$ in $X$ such that $y_\gamma \downarrow 0$ and that for
every 
$\gamma$, there exists $\al_0$ such that $\abs{x_\al - x} \leq y_\gamma$ for all
$\al \geq \al_0$.
If $X$ is a lattice ideal of the space of
all real-valued measurable functions $L^0(\Omega,\Sigma,\mu)$ on a measure space
$(\Omega,\Sigma,\mu)$, then order convergence of a sequence in $X$ is
equivalent to dominated almost everywhere convergence.
For obvious reasons, both theoretically and in applications, a generalization of
a.e.\ convergence to the abstract setting of
vector lattices is of much interest.
Motivated by this, the concept of  {\em unbounded order convergence} or {\em
uo-convergence} 
has recently been intensively studied in several papers
\cite{G:14,GTX:16,GX:14}.
A net $(x_\al)$  is said to {\em unbounded order
converge} ({\em uo-converge}) to $x\in X$ if $|x_\al-x| \wedge
y\xrightarrow{o}0$ for any $y \in X_+$.  In this case, we write
$x_\alpha\xrightarrow{uo}x$.
It is indeed easy to verify that if $X$ is a lattice ideal of  
$L^0(\Omega,\Sigma,\mu)$, then a sequence
$(f_n)$ in $X$ uo-converges to  $f\in X$ if and only if it converges
a.e.\ to $f$.

When $X$ is a Banach lattice, there is a well known duality theory associated
with order convergence.
For instance, a linear functional $\phi$ on $X$ is said to be {\em order
continuous} if
$\phi(x_\al) \to 0$ for any net $(x_\al)$ in $X$ that order converges to $0$.
The set $\Xn$ of all order continuous linear functionals on $X$ is called the
\emph{order continuous dual} of $X$. It is  a band
(i.e., order
closed lattice ideal) in $X^*$, and $X^* = \Xn$ if and only if
$X$ is \emph{order continuous}, i.e., $\|x_\al\|\to 0$ for any net
$(x_\al)$ in $X$ that order converges to $0$ (\cite[Theorem~2.4.2]{MN:91}).

One of the aims of the present paper is to develop a duality theory for
uo-convergence. Besides its intrinsic interest, a strong motivation derives
from the  representation theory of risk measures and convex functionals in
Mathematical Finance.  The following result was obtained by Gao and Xanthos in
\cite{GX:16a}.

\begin{theorem}[{\cite[Theorem~2.4]{GX:16a}}]\label{rep-orl-gx}
If an Orlicz space $L^\Phi$ is not equal to $L^1$, then the following
statements are equivalent for every proper (i.e., not identically $\infty$)
convex functional
$\rho:L^\Phi\rightarrow(-\infty,\infty]$.\begin{enumerate}
\item $\rho(f)=\sup_{g\in H^\Psi}\big(\int  fg-\rho^*(g) \big)$ for any
$f\in
L^\Phi $, where $H^\Psi$ is the conjugate Orlicz heart, and
$\rho^*(g)=\sup_{f\in L^\Phi}\big(\int  fg-\rho(f)\big)$ for any $g\in
H^\Psi$.
\item\label{o-dual-orlic2} $\rho(f)\leq \liminf_n \rho(f_n)$, whenever
$f_n\xrightarrow{a.e.} f$ and $(f_n)$ is norm  bounded in $L^\Phi $.
\end{enumerate}
\end{theorem}
\noindent Condition (2)  in Theorem~\ref{rep-orl-gx} suggests the following
definition.
Let $X$ be a Banach lattice.  A linear functional $\phi$ on $X$ is said to be
{\em boundedly uo-continuous} if 
$\phi(x_\al) \to 0$ whenever $x_\al\xrightarrow{uo} 0$ and $(x_\al)$ is {\em 
norm bounded}.
The set of all boundedly uo-continuous functionals on $X$ is called the {\em
uo-dual} of
$X$ and is denoted by 
$\Xu$.

The definition of boundedly uo-continuous functionals is closely connected to
that of order continuous functionals.
Recall that a net $(x_\alpha)$ in a vector lattice $X$ is said to be \emph{order
bounded} if there exists $x\in X_+$ such that $\abs{x_\alpha}\leq x$ for all
$\alpha$.
For any net $(x_\alpha)$ in $X$, it is easily seen that
$x_\alpha\xrightarrow{o}0$ if and only if $x_\alpha\xrightarrow{uo}0$ and a tail
of $(x_\alpha)$ is order bounded. 
Thus a linear functional $\phi$ on $X$ is order
continuous if and only if
$\phi(x_\alpha)\rightarrow 0$ whenever $x_\alpha\xrightarrow{uo}0$ and
$(x_\alpha)$ is \emph{order bounded}.

\medskip

In \S 2, we show  that $\Xu$ is precisely the order continuous
part of $\Xn$.  As a consequence, it is deduced that any Banach
lattice can have at most one order continuous predual up to lattice isomorphism,
namely, $\Xu$.
In \S 3, a characterization is given of precisely when $X = (\Xu)^*$.
In the final section, we apply these results to prove a theorem
(Theorem~\ref{general-rep}) on
 dual representation of convex functionals in the general setting of Banach
lattices
with order continuous preduals.  
Theorem~\ref{general-rep} generalizes Theorem~\ref{rep-orl-gx} to the abstract
setting.

\medskip

We adopt \cite{AB:78,AB:06} as standard references for unexplained terminology
and
facts on vector and Banach lattices. 
We will frequently use the following fact.
Recall that a sublattice $Y$ of a vector lattice $X$ is said to be
\emph{regular} if
$y_\alpha\downarrow
0$ in 
$Y$ implies $y_\alpha\downarrow 0$ in $X$.
In this case, by \cite[Theorem~3.2]{GTX:16}, for any net $(y_\alpha)$ in $Y$,
$$y_\alpha\xrightarrow{uo}0\mbox{ in }Y\quad
\iff\quad y_\alpha\xrightarrow{uo}0\mbox{ in }
X.$$
Ideals and order dense sublattices are regular. 
Order continuous norm closed sublattices of a Banach lattice are also regular.

\section{Characterization of the uo-dual}\label{uo-cont-dual}

\begin{definition}\label{def-uo-dual}
Let $X$ be a Banach lattice. A linear functional $\phi$ on $X$ is said to be
\emph{boundedly uo-continuous} if $\phi(x_\alpha)\rightarrow0$ for any
 norm bounded uo-null net $(x_\alpha)$ in $X$. The set of all boundedly
uo-continuous linear functionals on $X$ will be
called the \emph{unbounded order dual}, or \emph{uo-dual}, of $X$, and will be
denoted by
$X_{uo}^\sim$.
\end{definition}

The following proposition explains why the uo-dual is taken to be
the set of all {\em boundedly} uo-continuous functionals rather than just the
uo-continuous functionals.
Recall first that a vector $x> 0$ in an Archimedean vector lattice $X$ is an
{\em atom} if  
for any $u,v\in  [0,x ]$ with $u\wedge v=0$, either
$u=0$ or $v=0$.
In this case, the band generated by $x$ is one-dimensional, namely,
$\span\{x\}$. 
Moreover, the band projection $P$  from $X$ onto $\span\{x\}$ defined
by 
\[ Pz = \sup_n (z^+\wedge nx) - \sup_n(z^-\wedge nx)\]
exists, and there is a unique positive linear functional $\phi$ on $X$ such that
$Pz = \phi(z)x$ for
all $z\in X$.
We call $\phi$ the {\em coordinate functional} of the atom $x$.
Clearly, the span of any finite set of atoms is also a
projection band.
For any $\phi\in \Xn$, its \emph{null ideal} and \emph{carrier}  are the bands
of $X$
defined by $$N_\phi:=\{x\in
X:\abs{\phi}(\abs{x})=0\big\}\quad \mbox{ and }\quad C_\phi:=N_\phi^{\rm d},$$
respectively. Note that $\abs{\phi}$ acts as a strictly positive functional on
$C_\phi$: if $0< x\in
C_\phi$, then $\abs{\phi}(x) > 0$.

\begin{proposition}
Let $\phi$ be  a nonzero linear functional on an Archimedean vector lattice $X$
such that
$\phi(x_\alpha)\rightarrow 0$ whenever $x_\alpha\xrightarrow{uo}0$. Then $\phi $
is  a linear combination of the coordinate functionals of finitely many atoms.
\end{proposition}

\begin{proof}
Note first that $\phi$ is order continuous and is thus order bounded. We claim
that
$C_\phi$ cannot contain an infinite disjoint sequence of nonzero vectors.
Suppose otherwise that $(u_n)$ is   an infinite disjoint sequence of nonzero
vectors in $C_\phi$. Then $\abs{\phi}(\abs{u_n})>0$ for every $n\geq 1$.
By Riesz-Kantorovich Formula (\cite[Theorem~1.18]{AB:06}), 
there exists $v_n\in [-\abs{u_n},\abs{u_n}] $ such that $\phi(v_n)\neq 0$. Since
$\big(\frac{v_n}{\phi(v_n)}\big) $ is disjoint and thus uo-null
(\cite[Corollary~3.6]{GTX:16}),
$1=\phi\big(\frac{v_n}{\phi(v_n)}\big)\rightarrow0$, which is absurd.  
This proves the claim. It follows in particular that $C_\phi$ contains at most
finitely many disjoint atoms.
Let $B$ be the band generated by the atoms of $C_\phi$. It is a projection band
of $C_\phi$. If $B\neq C_\phi$, 
there would exist $0<x\in C_\phi$ such that $x\perp B$. Since $x$ is not an
atom, there exist
$u_1,y$ such that $0<u_1,y\leq
x$
and $u_1\perp y$. Clearly, $u_1,y\in C_\phi$.
Since $y\perp B$, $y$ is not an atom, and thus there exist $u_2,z$ such that
$0<u_2,z\leq
y$
and $u_2\perp z$. Clearly, $u_2,z\in C_\phi$. Repeating this process, we obtain
an infinite disjoint sequence of nonzero vectors in $C_\phi$, which contradicts
the claim.
Thus $C_\phi=B$ is generated by finitely
many
atoms. The desired result follows immediately.
\end{proof}

Since each order convergent net has a tail which is order bounded, and
therefore,
norm bounded, it is easy to see that $X_{uo}^\sim\subset X_n^\sim$.
The following theorem determines the precise position of $X_{uo}^\sim$ in
$X_n^\sim$. Recall first that the \emph{order continuous part}, $X^a$, of a
Banach lattice $X$ is given by 
$$X^{a}=\big\{x\in X: \mbox{ every disjoint sequence in }[0,\abs{x}]\mbox{ is
norm null}\big\}.$$
It is the largest norm closed ideal of $X$ which is order continuous in its own
right. 

\begin{theorem}\label{uo-dual-cont}
Let $X$ be a Banach lattice.
Then $X_{uo}^\sim$ is the order continuous part of $X_n^\sim$. Specifically, for
any
$\phi\in
X_n^\sim$, the following statements are equivalent:
\begin{enumerate}
\item\label{o-uo-cont1}  $\phi\in X^\sim_{uo}$,
\item\label{o-uo-cont1.5}  $\phi(x_n)\rightarrow 0$ for any norm bounded uo-null
sequence $(x_n)$ in $X$,
\item\label{o-uo-cont2}  $\phi(x_n)\rightarrow 0$ for any norm bounded disjoint
sequence $(x_n)$ in $X$,
\item\label{o-uo-cont3}  every disjoint sequence in
$[0,\abs{\phi}]$ is norm null.
\end{enumerate}
\end{theorem}

Observe  that, since $X_n^\sim$ is an ideal of $X^*$,  the interval
$[0,\abs{\phi}]$ is the same when taken in $X_n^\sim$
or 
in $X^*$ for any $\phi\in
X_n^\sim$.

\begin{proof}
The implications
\eqref{o-uo-cont1}$\implies$\eqref{o-uo-cont1.5}$\implies$\eqref{o-uo-cont2} are
obvious
since every disjoint sequence is uo-null by \cite[Corollary~3.6]{GTX:16}.
Assume now that \eqref{o-uo-cont2} holds. By Riesz-Kantorovich Formula, it
can be easily verified that $\abs{\phi}(\abs{x_n})\rightarrow 0$ for any
disjoint
sequence $(x_n)$ in the closed unit  ball $B_X$. Applying
\cite[Theorem~4.36]{AB:06} to the seminorm
$\abs{\phi}(\abs{\cdot})$, the identity operator and $B_X$, we have that, for
any
$\varepsilon>0$,
there exists $u\in X_+$ such that $$\sup_{x\in
B_X}\abs{\phi}\big(\abs{x}-\abs{x}\wedge u\big)=\sup_{x\in
B_X}\abs{\phi}\big((\abs{x}-u)^+\big)\leq \varepsilon.$$
If $(x_\alpha)$ is a uo-null net in $B_X$, then $\abs{x_\alpha}\wedge
u\xrightarrow{o}0$, so that $\abs{\phi}(\abs{x_\alpha}\wedge u)\rightarrow 0$.
Therefore, $$\limsup_\alpha\abs{\phi}(\abs{x_\alpha})\leq \varepsilon.$$
By arbitrariness of $\varepsilon$, we have $$\abs{\phi(x_\alpha)}\leq
\abs{\phi}(\abs{x_\alpha})\rightarrow0.$$
This proves that $\phi\in X_{uo}^\sim$.
Hence, \eqref{o-uo-cont2}$\implies$\eqref{o-uo-cont1}. 
The equivalence of \eqref{o-uo-cont2} and \eqref{o-uo-cont3} follows from
\cite[Theorem~2.3.3]{MN:91} with $A=B_X$ and $B=[-\abs{\phi},\abs{\phi}]$, where
the order interval is regarded as taken in $X^*$.
\end{proof}

It follows, in particular, that $\Xu$ is a norm closed ideal of $\Xn$ and of
$X^*$ and is an order continuous Banach lattice itself.

\begin{example}
$(\ell^1)_{uo}^\sim=(\ell^\infty)^a=c_0$; $(c_0)^\sim_{uo}=(\ell^1)^a=\ell^1$;
$(\ell^\infty)_{uo}^\sim=(\ell^1)^a=\ell^1$.
\end{example}

\begin{corollary}[{\cite[Theorem 5]{W:77}}]\label{wick}
For a Banach lattice $X$, $X^\sim_{uo}=X^*$ iff $X$ and $X^*$ are both order
continuous iff every norm bounded uo-null net is weakly null.
\end{corollary}

\begin{proof}
$X^\sim_{uo}=X^*$ means precisely that every norm  bounded
uo-null net in $X$ is weakly null. Thus it suffices to prove the first
``iff''.
Assume that $X$ and $X^*$ are both order continuous. Then  $X_n^\sim=X^*$, and
$\Xn$ is
order continuous.
By Theorem~\ref{uo-dual-cont}, $X_{uo}^\sim=X_n^\sim=X^*$.
Conversely, assume that $X^\sim_{uo}=X^*$. In view of $X_{uo}^\sim\subset
X_n^\sim \subset X^*$, it follows that
$X_{uo}^\sim=X^*=X_n^\sim$.
By the second equality, $X$ is order continuous. By the first equality and 
Theorem~\ref{uo-dual-cont} again, $X^*$ is order
continuous.
\end{proof}

Recall that $X$ is identified as a norm closed sublattice of $X^{**}$ by the
evaluation mapping $$j:X\rightarrow X^{**};\;\;j(x)(\phi)=\phi(x),$$
for any $x\in X$ and $\phi\in X^*$.
In fact, $j(x)\in (X^*)_n^\sim$. Indeed, if
$\phi_\alpha\downarrow0$ in $X^*$, then
$j(x)(\phi_\alpha)=\phi_\alpha(x)\rightarrow 0$ by \cite[Theorems~1.18]{AB:06}.
Thus $j(x)$ is order continuous on $X^*$ by \cite[Theorem~1.56]{AB:06}.
We identify $X$ as a norm closed sublattice of
$(X^*)_n^\sim$.

\begin{proposition}\label{dual-uo-cont}
For a Banach lattice $X$, $X\subset (X^*)_{uo}^\sim$ iff $X$ is order
continuous iff $X=(X^*)_{uo}^\sim$.  In particular, a Banach lattice has at most
one order continuous
lattice isomorphic predual, up to lattice isomorphism.
\end{proposition}

\begin{proof}
It is well-known (and easy to verify) that a norm closed sublattice of an order
continuous Banach lattice is order continuous in its own right. 
Thus by Theorem~\ref{uo-dual-cont}, if $X\subset (X^*)_{uo}^\sim$, then $X$
is
order continuous.
Suppose now that $X$ is order continuous. 
It follows from $X\subset (X^*)_n^\sim$ that $X$ is contained in the order
continuous part,
$(X^*)_{uo}^\sim$, of $(X^*)_n^\sim$. 
Moreover, by \cite[Theorem~1.70]{AB:06}, $X$ is order dense in
$(X_n^\sim)_n^\sim=(X^*)_n^\sim$, 
and therefore, in $(X^*)_{uo}^\sim$. 
This, together with order continuity of $(X^*)_{uo}^\sim$, implies that $X$ is
also
norm dense in
$(X^*)_{uo}^\sim$.
Therefore, since $X$ is norm closed in $(X^*)_{uo}^\sim$, 
$X=(X^*)_{uo}^\sim$.

For the last assertion, simply note that if $X$ is order continuous and $X^*$ is
 lattice isomorphic to $Y$, then  $X =(X^*)_{uo}^\sim$ is lattice
isomorphic to  $Y^\sim_{uo}$.
\end{proof}

\section{Banach lattices with order continuous preduals}\label{predual}

In this section, we characterize the Banach lattices which have order
continuous preduals.
We begin with the following proposition 
which is of independent interest and generalizes \cite[Theorem~4.7]{GX:14}.
Recall first that a Banach lattice is said to be
\emph{monotonically complete} if every norm bounded positive  increasing net has
a
supremum.
A net $(x_\al)$ in a vector lattice $X$ is said to be {\em order Cauchy},
respectively, {\em uo-Cauchy} if the double net $(x_\al - x_\beta)_{\al,\beta}$
order converges to $0$, respectively, uo-converges to $0$ in $X$.
Following \cite{GTX:16}, we say that a Banach lattice is \emph{boundedly
uo-complete} if
every norm
bounded uo-Cauchy net is uo-convergent.

\begin{proposition}\label{buoc-mc}A monotonically complete Banach lattice $X$ is
boundedly
uo-complete. The converse is true if $X_n^\sim$ separates points of $X$.
\end{proposition}

\begin{proof}
Let $X$ be a monotonically complete Banach lattice.
Recall from \cite[Proposition~2.4.19(i)]{MN:91} that $X$ is order complete and
admits a constant $C$ such that
$$\norm{x}\leq
C\sup_\alpha\norm{x_\alpha}\;\;\mbox{ whenever } 0\leq x_\alpha\uparrow x\mbox{
in }X.$$ 
Thus if $x_\alpha\xrightarrow{o}x$
in $X$, 
then $y_\alpha:=\inf_{\beta\geq \alpha} \abs{x_\beta}\uparrow \abs{x}$,
and hence $$\norm{x}\leq C\sup_\alpha\norm{y_\alpha}\leq
C\sup_\alpha\norm{x_\alpha}.$$
Now let $(x_\alpha)$ be any norm bounded uo-Cauchy net in $X$. By considering
the positive and negative parts, respectively, we may assume that $x_\alpha\geq
0$ for each $\alpha$.
For each $y\in X_+$, since $\abs{x_\alpha\wedge y-x_{\alpha'}\wedge y}\leq
\abs{x_\alpha-x_{\alpha'}}\wedge y$, the net $(x_\alpha \wedge y)$ is order
Cauchy and hence order converges to some $u_y \in X_+$.
The net $(u_y)_{y\in X_+}$ is directed upwards, and $$\|u_y\| \leq
C\sup_\alpha\norm{x_\alpha\wedge y}\leq
C\sup_\alpha\|x_\alpha\|$$ for all $y\in X_+$, by the preceding observation.
Since $X$ is monotonically complete, $(u_y)$ increases to an element $u \in X$.
Fix $y \in X_+$.  For any $\alpha, \alpha'$, define 
\[ x_{\alpha,\alpha'} = \sup_{\beta\geq \alpha, \beta'\geq \alpha'}|x_\beta -
x_\beta'| \wedge y.\]
Since $(x_\alpha)$ is uo-Cauchy, $x_{\alpha,\alpha'} \downarrow 0$.
Also, for any $z\in X_+$ and any $\beta \geq \alpha$, $\beta' \geq \alpha'$,
\[ |x_\beta \wedge z - x_{\beta'}\wedge z| \wedge y\leq x_{\alpha,\alpha'} \]
Taking order limit first in $\beta'$ and then supremum over $z$ in $X_+$, we
obtain $\abs{x_\beta -
u}\wedge y \leq x_{\alpha,\alpha'}$ for any $\beta \geq\alpha$.
This implies that $(x_\alpha)$ uo-converges to $u$.

For the second assertion, suppose that $X$ is
boundedly uo-complete and $X_n^\sim$ separates points of $X$. Let $(x_\alpha)$
be a norm
bounded increasing positive net in
$X$. 
Consider the evaluation mapping $$j:X\rightarrow
(X_n^\sim)_n^\sim;\;\;j(x)(\phi)=\phi(x),$$
for any $x\in X$ and $\phi\in \Xn$. 
By \cite[Theorem~1.70]{AB:06}, $j$ is an (into) vector lattice isomorphism, and
$j(X)$ is an order dense, 
in particular regular,
sublattice in $(X_{n}^\sim)_n^\sim$. 
Since $\norm{j}\leq 1$, $(j(x_\alpha))$ is an
increasing norm bounded positive net in
$(X_{n}^\sim)_n^\sim$. By \cite[Proposition~2.4.19]{MN:91}, $(X_n^\sim)^\sim_n$
is monotonically complete, so that there exists $x^{**}\in (X_{n}^\sim)_n^\sim$
such that
$$j(x_\alpha)\uparrow x^{**}\;\;\mbox{ in }\;(X_n^\sim)_n^\sim.$$
In particular,
$(j(x_\alpha))$ is order Cauchy, and therefore uo-Cauchy, in
$(X_{n}^\sim)_n^\sim$. It
follows from \cite[Theorem~3.2]{GTX:16} that
$(j(x_\alpha))$ is uo-Cauchy in $j(X)$.
Since $j$ is one-to-one and onto $j(X)$, $(x_\alpha)$ is
uo-Cauchy in $X$. Let $x$ be the uo-limit of $(x_\alpha)$ in $X$. It is easy to
check
that $x_\alpha\uparrow x$ in $X$.  
\end{proof}

We need two technical lemmas in preparation for the main result of
this section.

\begin{lemma}\label{uo-tec-1}
Let $X$ be a Banach lattice such that $X_n^\sim$ separates points of $X$, and
let $I$ be an ideal of $X_n^\sim$. 
Then $I$ is order dense in $X_n^\sim$ iff it separates points of $X$.
\end{lemma}

\begin{proof}
The ``only if'' part is clear. For the ``if'' part, suppose that $I$ separates
points of $X$ but is not order dense in $X_n^\sim$. 
By \cite[Theorem~1.36]{AB:06}, there exists $0<\phi\in X_n^\sim$ such that
$\phi\perp
I$.
By \cite[Theorem~1.67]{AB:06}, it follows that $C_\phi\subset \cap_{\psi\in
I}N_{\psi}=\{0\}$, where
the last equality holds because $I$ separates points of $X$. 
Thus, $C_\phi=\{0\}$, and $N_\phi=C_\phi^{\rm d}=X$, implying that $\phi=0$, a
contradiction.
\end{proof}

The proof of the following lemma is inspired by that of
\cite[Theorem~1.65]{AB:06}. We provide the details for the convenience of the
reader.

\begin{lemma}\label{uo-tec-2}
If $Y$ is a norm closed order dense ideal of a Banach lattice $Z$, then
$Z_n^\sim$ is lattice isometric to $Y^\sim_n$ via the restriction mapping.
\end{lemma}

\begin{proof}
We put the restriction mapping $R:Z_n^\sim\rightarrow Y_n^\sim$ by
$R\phi=\phi_{|Y}$
for any $\phi\in Z_n^\sim$.
Since $Y$ is an ideal of
$Z$, each order null net in $Y$ is also order null in $Z$. Thus 
$R\phi\in Y_n^\sim$ for any $\phi\in Z_n^\sim$, and $R$ indeed maps $Z_n^\sim$
into
$Y_n^\sim$. 
For any $z\in Z$, by order denseness of $Y$, we can take two nets $(u_\alpha) $
and $(v_\alpha)$ in $Y$ such that $0\leq u_\alpha\uparrow z^+$ and $0\leq
v_\alpha\uparrow z^-$ in $Z$. 
Put $y_\alpha=u_\alpha-v_\alpha$. Then $\abs{y_\alpha}\leq \abs{z}$ and
$y_\alpha\xrightarrow{o} z$ in $Z$. It follows that
$$\norm{\phi}=\sup_{z\in
B_Z}\abs{\phi(z)}=\sup_{y\in B_Y}\abs{\phi(y)}=\norm{R\phi}$$
for any $\phi\in Z_n^\sim$, where $B_Z$ and $B_Y$ are the closed unit balls of
$Z$ and $Y$, respectively. 
This proves that $R$ is an isometry.

We now show that $R$ is surjective. Pick any $\psi\in Y_n^\sim$. 
By considering $\psi^\pm$, we may assume that $\psi\geq 0$.
For any $z\in Z_+$, put
$$\phi(z):=\sup\big\{\psi(y):y\in Y,0\leq y\leq z\big\};$$
the supremum is finite because $0\leq \psi(y)\leq \norm{\psi}\norm{y}\leq
\norm{\psi}\norm{z}$.
It is clear that  $ \phi(z_1)+\phi(z_2)\leq \phi(z_1+z_2)$ for any $z_1,z_2\in
Z_+$.
Conversely, by Riesz Decomposition Theorem, it is easy to see that
$\phi(z_1+z_2)\leq \phi(z_1)+\phi(z_2)$.
Thus $$ \phi(z_1)+\phi(z_2)= \phi(z_1+z_2)$$ for any $z_1,z_2\in Z_+$.
By Kantorovich Extension Theorem (\cite[Theorem~1.10]{AB:06}), $\phi$ determines
a linear functional on $Z$
by setting $$\phi(z)=\phi(z^+)-\phi(z^-)$$ for any $z\in Z$.
We show that $\phi\in Z_n^\sim$.
Take any $0\leq z_\alpha\uparrow z$ in $Z$. Put $D=\{y\in Y_+:y\leq
z_\alpha\mbox{ for some }\alpha\}$. 
Then $D\uparrow z$, and $\sup_\alpha\phi(z_\alpha)=\sup \psi(D)\leq \phi(z)$.
For any $y\in Y_+$ with $y\leq z$, $D\wedge y\uparrow z\wedge y=y$ in $Z$ and
thus also in $Y$. Therefore, $\psi(y)=\sup\psi(D\wedge y)\leq \sup\psi(D)$.
Hence, $\phi(z)\leq \sup \psi(D)=\sup_\alpha\phi(z_\alpha)$, and consequently,
$\phi(z)=\sup_\alpha\phi(z_\alpha)$.
It follows that $\phi\in Z_n^\sim$. Clearly, $R\phi=\psi$.

Finally, for any $\phi\in Z_n^\sim$, it is clear that $R\phi\geq 0$ iff
$\phi\geq 0$. Therefore, $R$ is a lattice isomorphism, by
\cite[Theorem~2.15]{AB:06}.
\end{proof}

The next theorem is the main result of this section.  It gives a
characterization of  Banach lattices $X$ that are canonically isomorphic to
$(\Xu)^*$.
In light of Proposition~\ref{dual-uo-cont}, this is equivalent to characterizing
when $X$ has
an order continuous lattice isomorphic predual.

\begin{theorem}\label{predual-cha}Let $X$ be a Banach lattice such that
$X_{uo}^\sim$ separates points of $X$. The following statements are equivalent.
\begin{enumerate}
\item\label{predual-cha1} The mapping $j:X\rightarrow (X_{uo}^\sim)^*$ is a
surjective
lattice isomorphism, where $j(x)(\phi)=\phi(x)$ for any $x\in X$ and any
$\phi\in
X_{uo}^\sim$.
\item\label{predual-cha2} $B_X$ is relatively $\sigma(X,X_{uo}^\sim)$-compact in
$X$.
\item\label{predual-cha3} $X$ is boundedly uo-complete.
\item\label{predual-cha4} $X$ is monotonically complete.
\end{enumerate}
\end{theorem}

\begin{proof}
The implication \eqref{predual-cha1}$\implies$\eqref{predual-cha2} is
immediate by Banach-Alaoglu Theorem. 

Assume that \eqref{predual-cha2} holds.
Clearly, $j$ is one-to-one and $\norm{j}\leq 1$. 
Recall that $X_{uo}^\sim$ is an ideal of $X_n^\sim$. Note
also that $(X_{uo}^\sim)^*=(X_{uo}^\sim)_n^\sim$ since $X_{uo}^\sim$ is order
continuous by Theorem~\ref{uo-dual-cont}. Thus by
\cite[Theorem~1.70]{AB:06} again, $j$ is a bijective vector lattice isomorphism
between
$X$ and
$j(X)$, and $j(X)$ is a regular sublattice in $(X_{uo}^\sim)^*$.
Let $(x_\alpha)$ be a norm bounded uo-Cauchy net in $X$. Then $(j(x_\alpha))$ is
uo-Cauchy in $j(X)$, and thus in $(X_{uo}^\sim)^*$, by
\cite[Theorem~3.2]{GTX:16}. Since $(j(x_\alpha))$ is norm bounded in
$(X_{uo}^\sim)^*$, by \cite[Theorem~2.2]{G:14}, there exists $x^{**}\in
(X_{uo}^\sim)^*$ such that
$$j(x_\alpha)\xrightarrow[\sigma((X_{uo}^\sim)^*,X_{uo}^\sim)]{uo} x^{**}\mbox{
in
}(X_{uo}^\sim)^*.$$
On the other hand, by the assumption, there exist $x\in X$
and a subnet $(x_{\beta})$ of $(x_\alpha)$ such that
$x_\beta\xrightarrow{\sigma(X,X_{uo}^\sim)}x$, or equivalently,
$$j(x_\beta)\xrightarrow{\sigma((X_{uo}^\sim)^*,X_{uo}^\sim)} j(x).$$
It follows
that $x^{**}=j(x)\in j(X)$, and  $j(x_\alpha)\xrightarrow{uo} j(x)$ in
$(X_{uo}^\sim)^*$. By \cite[Theorem~3.2]{GTX:16}
again, we have $j(x_\alpha)\xrightarrow{uo} j(x)$ in $j(X)$, and
consequently,
$x_\alpha\xrightarrow{uo}x$ in $X$. This proves
\eqref{predual-cha2}$\implies$\eqref{predual-cha3}.
The implication \eqref{predual-cha3}$\implies$\eqref{predual-cha4} follows
from Proposition~\ref{buoc-mc}.

Suppose that \eqref{predual-cha4} holds. By \cite[Theorem~2.4.22]{MN:91},
$X$ is
lattice isomorphic to
$(X_n^\sim)_n^\sim$ via the evaluation mapping $e : X \to \Xnn$. Since $\Xu$ is
order
continuous, $(\Xu)^* = (\Xu)_n^\sim$.
Let $R: \Xnn\to (\Xu)_n^\sim=(\Xu)^*$ be the restriction map $R\phi =
\phi_{|\Xu}$. By Lemma~\ref{uo-tec-1}, $X_{uo}^\sim$ is order dense in
$X_n^\sim$.  Thus by Lemma~\ref{uo-tec-2}, $R$ is a lattice isometry from 
$\Xnn$ onto $(\Xu)^*$.
This proves \eqref{predual-cha4}$\implies$\eqref{predual-cha1} since $j = Re$.
\end{proof}

We point out that the mapping $j$ in Theorem~\ref{predual-cha} is an isometry
iff $X_{uo}^\sim$ is norming on $X$, 
iff $X_n^\sim$ is norming, iff $0\leq
x_\alpha\uparrow x$ implies $\norm{x_\alpha}\uparrow \norm{x}$
(\cite[Theorem~2.4.21]{MN:91}), iff the closed
unit ball $B_X$ is order closed.

\section{Dual representations of convex functionals}\label{representation}

In this section, we apply the general theory pertaining to the
uo-dual developed in the previous sections to the representations of convex
functionals on a Banach lattice $X$ with respect to the duality $(X,\Xu)$.
The main motivation is Theorem~\ref{general-rep}, which gives a formulation of
Theorem~\ref{rep-orl-gx} in
the general setting of Banach lattices.
The main tool is the next theorem, whose conclusion  should be compared with the
$C$-property introduced in \cite{BF:10}.

\begin{theorem}\label{uo-exist}
Let $X$ be a vector lattice and $I$ be an ideal of $X^\sim$
containing a strictly positive order continuous functional
$\phi$ on $X$. If $C$ is a convex subset of $X$ and $x \in
\overline{C}^{\sigma(X,I)}$, then there exists a sequence
$(x_n)$ in $C$ such that $x_n \xrightarrow{uo} x$ in $X$.
\end{theorem}

\begin{proof}
By Kaplan's Theorem (\cite[Theorem~3.50]{AB:06}), the topological dual of $X$
under $\abs{\sigma}(X,I)$ is precisely $I$. Thus by Mazur's Theorem, we have 
$$
\overline{C}^{{\sigma}(X,I)}=\overline{C}^{\abs{\sigma}(X,I)}.$$
Consequently, there exists a net
$(x_\alpha) $ in $ C$ such that $x_\alpha
\xrightarrow{{\abs{\sigma}(X,I)}} x$. 
In particular, $\phi(\abs{x_\alpha-x})\rightarrow 0$. 
Choose $(\alpha_n)$ such that $$\phi(\abs{x_{\alpha_n}-x})\leq
\frac{1}{2^n}$$
for
each $n\geq 1$. 
For the sake of convenience,  write $x_n:=x_{\alpha_n}$. 
It remains to be shown that $x_{n}\xrightarrow{uo}x$ in $X$. 
For any $y\in X$, put $\norm{y}_L=\phi(\abs{y})$. 
Then $\norm{\cdot}_L$ is a norm on $X$, and the norm
completion $\widetilde{X}$ of $(X,\norm{\cdot}_L)$ is an $L^1$-space, in which
$X$ sits as a
regular sublattice; cf.~\cite[Section~4]{GTX:16}. 
Since $\norm{x_n-x}_L\leq
\frac{1}{2^n}$, it follows that $\sum_1^\infty\abs{x_n-x}$ converges in
$\widetilde{X}$. As $$\abs{x_n-x}\leq
\sum_{n}^\infty\abs{x_k-x}\downarrow0 \mbox{ in }\widetilde{X},$$
we have $x_n\xrightarrow{o}x$
in $\widetilde{X}$, so that $x_n\xrightarrow{uo}x$ in $\widetilde{X}$,
and in $X$ by \cite[Theorem~4.1]{GTX:16}.
\end{proof}

The following is an interesting special case of Theorem~\ref{uo-exist}.
Recall that a Banach function space over a
$\sigma$-finite measure space admits strictly positive order
continuous functionals (see, e.g., \cite[Proposition~5.19]{GTX:16}).

\begin{corollary}
Let $X$ be a Banach function space over a $\sigma$-finite
measure space.
For any convex subset $C$ of $X$ and $f \in
\overline{C}^{\sigma(X,X_n^\sim)}$, there exists a sequence
$(f_n)$ in $C$ such that $f_n\xrightarrow{a.e.}f$.
\end{corollary}

In line with order closures, we define the \emph{bounded uo closure} of a set
$C$ in a Banach lattice $X$ by
$$\overline{C}^{buo}:=\big\{x\in X: \;x_\alpha
\xrightarrow{uo}x\mbox{ in }X\mbox{ for some bounded net }(x_\alpha)\mbox{ in
}C\big\}.$$ 
We similarly
define the \emph{bounded uo sequential closure} $\overline{C}^{seq-buo}$ by
replacing
nets with sequences.
We say that $C$ is \emph{boundedly uo closed}, respectively, \emph{boundedly uo
sequentially closed} in $X$, if
$C=\overline{C}^{buo}$, respectively, $C=\overline{C}^{seq-buo}$. It is clear
that every $\sigma(X,\Xu)$-closed set is boundedly uo (sequentially) closed, and
for any set
$C$,
$$\overline{C}^{seq-buo}\subset \overline{C}^{buo}\subset
\overline{C}^{\sigma(X,X_{uo}^\sim)}.$$
It is natural to wonder when the reverse inclusions hold. Namely, if $x\in
 \overline{C}^{\sigma(X,X_{uo}^\sim)}$, does there always exist a norm bounded
net
(sequence)
in $C$ that uo-converges to $x$?  This is clearly a strengthened version of
Theorem~\ref{uo-exist}.

\begin{proposition}\label{buo-closure}
Let $X$ be a $\sigma$-order complete Banach lattice. The following statements
are equivalent. 
\begin{enumerate}
 \item\label{last1} $\overline{C}^{buo}=\overline{C}^{\sigma(X,\Xu)}$ for every
convex set
$C$.
\item\label{last2} $X$ and $X^*$ are both order continuous. 
\end{enumerate}
\end{proposition}

\begin{proof}
Suppose that  $X$ and $X^*$ are both order continuous. Let $C$ be a convex set
in $X$. It suffices to show that $\overline{C}^{\sigma(X,\Xu)}\subset
\overline{C}^{buo}$.
By Corollary~\ref{wick}, $\Xu=X^*$,  and  $\sigma(X,\Xu)$ is the
weak topology on $X$. Thus by Mazur's Theorem,
$\overline{C}^{\sigma(X,\Xu)}=\overline{C}^{\norm{\cdot}}$.
Since  every norm
convergent sequence admits a subsequence order converging to the same
limit (see, e.g., \cite[Lemma~3.11]{GX:14}), it follows that 
$\overline{C}^{\sigma(X,\Xu)}\subset
\overline{C}^{buo}$.

Conversely, suppose first that $X$ is not order continuous.
Then $X$ has a lattice isomorphic copy of $\ell^\infty$. 
A close look at the proof of \cite[Theorem~4.51]{AB:06} shows that the copy of
$\ell^\infty$ can be chosen to be regular in $X$. For convenience, we simply
assume that $\ell^\infty\subset X$.
By Ostrovskii's Theorem (cf.~\cite[Theorem~2.34]{HZ:07}), there exist a subspace
$W$ of $\ell^\infty$ and
$w\in \overline{W}^{\sigma(\ell^\infty,\ell^1)}$ such that $w$ is not the
$\sigma(\ell^\infty,\ell^1)$-limit of any sequence in $W$. 
By \cite[Theorem~3.2]{GTX:16},
since $\ell^\infty$ is regular in $X$, every uo-null net in $\ell^\infty$ is
also uo-null in $X$. Thus the restriction of each $\phi\in
\Xu$
to $\ell^\infty$ belongs to $(\ell^\infty)_{uo}^\sim=\ell^1$. 
Therefore, $w\in \overline{W}^{\sigma(X,\Xu)}$.
Suppose that $W$ admits a bounded net $(x_\alpha)$ such that
$x_\alpha\xrightarrow{uo}w$
in $X$.
By \cite[Theorem~3.2]{GTX:16} again, $x_\alpha\xrightarrow{uo}w$ in
$\ell^\infty$.
Clearly, we can select countably many $\alpha_n$'s such that 
$x_{\alpha_n}\rightarrow w$ coordinatewise. Since $(x_{\alpha_n})$ is bounded
in
$\ell^\infty$, it follows that
$x_{\alpha_n}\xrightarrow{\sigma(\ell^\infty,\ell^1)}w$, contradicting the
choice of $w$.

Suppose now that $X^*$ is not order continuous. Then by
\cite[Theorem~4.69]{AB:06}, $X$ has a lattice
copy of $\ell^1$. Since $\ell^1$ is order continuous,
it is easily seen that it is regular in $X$. By Ostrovskii's Theorem again,
there exist a subspace
$W$ of $\ell^1$ and
$w\in \overline{W}^{\sigma(\ell^1,c_0)}$ such that $w$ is not the
$\sigma(\ell^1,c_0)$-limit of any sequence in $W$. Since
$(\ell^1)^\sim_{uo}=c_0$, the same arguments as above also lead to a
contradiction.
\end{proof}

While Proposition~\ref{buo-closure} asserts that the bounded uo closure
and the $\sigma(X,\Xu)$-closure of a general convex set rarely coincide, the
situation
is different for the coincidence of bounded uo closedness and
$\sigma(X,\Xu)$-closedness of convex sets, where 
an analogue of the Krein-Smulian property  arises naturally.
For any $x\in X$, put $$\norm{x}_{uo}:=\sup_{\phi\in
X_{uo}^\sim,\norm{\phi}=1}\abs{\phi(x)}.$$
Then the set $$B:=\big\{x\in
X:\norm{x}_{uo}\leq
1\big\}$$ is $\sigma(X,X_{uo}^\sim)$-closed, and is thus boundedly uo closed. We
say that
$\sigma(X,X_{uo}^\sim)$ satisfies the \emph{Krein-Smulian property} if every
convex set of $X$ is $\sigma(X,X_{uo}^\sim)$-closed whenever its
intersections with
$kB$ are $\sigma(X,X_{uo}^\sim)$-closed for all $k\geq 1$.

\begin{corollary}\label{uo-closed-all}
Let $X$ be a Banach lattice that admits a strictly positive boundedly 
uo-continuous functional. Suppose that $X_{uo}^\sim$ is isomorphically norming
on $X$. For any convex set $C$ in $X$, the following are equivalent:
\begin{enumerate}
 \item\label{uo-closed-all1} $C$ is boundedly uo (sequentially) closed,
\item\label{uo-closed-all2} $C$ is $\sigma(X,X_{uo}^\sim)$-sequentially closed,
\item\label{uo-closed-all3} $C\cap kB$ is $\sigma(X,X_{uo}^\sim)$-closed for all
$k\geq 1$.
\end{enumerate}
Consequently, the following statements are equivalent to each other:
\begin{enumerate}
 \item[(1')] Every boundedly uo (sequentially) closed convex set $C$  is
$\sigma(X,X_{uo}^\sim)$-closed,
\item[(2')] Every $\sigma(X,X_{uo}^\sim)$-sequentially closed convex set is 
$\sigma(X,X_{uo}^\sim)$-closed,
\item[(3')] $\sigma(X,X_{uo}^\sim)$ satisfies the {Krein-Smulian property}.
\end{enumerate}
\end{corollary}

\begin{proof}
Since $B$ is norm bounded, the implication
\eqref{uo-closed-all1}$\implies$\eqref{uo-closed-all3} is immediate by
Theorem~\ref{uo-exist}. 
The implication \eqref{uo-closed-all3}$\implies$\eqref{uo-closed-all1} is
clear.
The implication
\eqref{uo-closed-all3}$\implies$\eqref{uo-closed-all2} follows because every
$\sigma(X,X_{uo}^\sim)$-convergent sequence is bounded. Clearly,
\eqref{uo-closed-all2}
implies the sequential version of \eqref{uo-closed-all1}.
This proves that \eqref{uo-closed-all1}, \eqref{uo-closed-all2} and
\eqref{uo-closed-all3} are equivalent. The equivalence of (1'), (2') and (3')
now follows immediately.
\end{proof}

\begin{proposition}\label{uo-rep}
Let $X$ be a Banach lattice that admits a strictly positive boundedly
uo-continuous functional. Suppose that $X_{uo}^\sim$ is isomorphically norming
on $X$ and
$\sigma(X,X_{uo}^\sim)$ satisfies
the {Krein-Smulian property} (these conditions are satisfied when $X$ is
monotonically
complete). Then for every proper convex functional
$\rho:X\rightarrow(-\infty,\infty]$, the following statements are equivalent to
each other:
\begin{enumerate}
 \item $\rho(x)=\sup_{\phi\in X_{uo}^\sim}\big(\phi(x)-\rho^*(\phi)\big)$ for
any $x\in X$, where $\rho^*(\phi)=\sup_{x\in X}\big(\phi(x)-\rho(x)\big)$ for
any $\phi\in \Xu$,
\item $\rho(x)\leq \liminf_\alpha\rho(x_\alpha)$ for any norm bounded net
$(x_\alpha)$ uo-converging to $x$, 
\item $\rho(x)\leq \liminf_n\rho(x_n)$ for any norm bounded sequence $(x_n)$
uo-converging to $x$.
\end{enumerate}
\end{proposition}

\begin{proof}
By Fenchel-Moreau Formula (\cite[Theorem~1.11]{B:11}), (1) is
equivalent to that $\rho$ is $\sigma(X,X_{uo}^\sim)$-lower semicontinuous, i.e.,
the level set $$C_\lambda:=\big\{x:\rho(x)\leq \lambda\big\}$$ is
$\sigma(X,X_{uo}^\sim)$-closed for any $\lambda\in\R$.
By Corollary~\ref{uo-closed-all}, this is equivalent to that $C_\lambda$ is
boundedly  uo (sequentially) closed for all $\lambda\in\R$. A straightforward
computation shows
that this last condition is equivalent to  (2), respectively, (3).

Finally, note that, since $X$ has a strictly positive uo-continuous functional,
$X_{uo}^\sim$ separates points of $X$. Thus if $X$ is
monotonically complete, then it is lattice isomorphic to $(X_{uo}^\sim)^*$, by
Theorem~\ref{predual-cha}. Consequently, $X_{uo}^\sim$ is
isomorphically norming on $X$ and
$\sigma(X,X_{uo}^\sim)$ satisfies
the {Krein-Smulian property}.
\end{proof}

\begin{theorem}\label{general-rep}Let $Y$ be an order continuous Banach lattice
with weak units,
and let $X=Y^*$. For any proper convex functional
$\rho:X\rightarrow(-\infty,\infty]$, the following statements are equivalent:
 \begin{enumerate}
 \item $\rho(x)=\sup_{y\in Y}(\la x,y\ra
-\rho^*(y))$ for any $x\in X$, where
$\rho^*(y)=\sup_{x\in X}(\la x,y\ra -\rho(x))$ for any $y\in
Y$,
\item $\rho(x)\leq \liminf_\alpha \rho(x_\alpha)$ for any norm bounded net
$(x_\alpha)$ uo-converging to $x$, 
\item $\rho(x)\leq \liminf_n \rho(x_n)$ for any norm bounded sequence $(x_n)$
uo-converging to $x$.
\end{enumerate}
\end{theorem}
 
\begin{proof}
By Proposition~\ref{dual-uo-cont}, $Y =\Xu$. 
Thus $\Xu$ norms $X$ and $\sigma(X,\Xu)$, being the weak$^*$ topology on $X$,
satisfies the Krein-Smulian
property.
Let $y \in Y_+$ be a weak unit in $Y$.  Since $Y$ is order continuous, $z\wedge
ny\rightarrow z$ in norm for any $z\in Y_+$.
It follows that if $x\in X_+$ and $\la x,y\ra =0$, then $\la x,z\ra =0$ for all
$z\in Y_+$, and hence $x =0$.
Therefore, $y$ acts as a strictly positive functional on $X$.
The result now follows from Proposition~\ref{uo-rep}.
\end{proof}

As was mentioned in the Introduction, a prime motivation for  Theorem
\ref{general-rep} comes from Theorem \ref{rep-orl-gx}, which is concerned with
Orlicz spaces.
In the discussion below, we   confine ourselves to probability  spaces as
they are the most important in  applications to risk measures in Mathematical
Finance.
Fix a probability space $(\Omega,\Sigma,\mathbb{P})$. Assume that it admits  an
infinite disjoint sequence of measurable sets of positive probabilities; for
otherwise the resulting spaces will be finite-dimensional.
An {\em Orlicz function} is a convex nondecreasing function $\Phi:[0,\infty)\to
\R$, not identically $0$,  such that $\Phi(0) =0$. 
For an Orlicz function $\Phi$, the {\em Orlicz space} 
$L^\Phi:=L^\Phi(\Omega,\Sigma,\mathbb{P})$ is the space of all measurable
functions $f$ such that 
\[ \|f\|_\Phi = \inf\Big\{\lambda> 0:
\int\Phi\Big(\frac{|f(\omega)|}{\lambda}\Big)\,\mathrm{d}\mathbb{P} \leq
1\Big\}\]
is finite.  The functional $\|\cdot\|_\Phi$ is a complete norm on $L^\Phi$,
called the {\em Luxemburg norm}.
The subspace of $L^\Phi$ consisting of all $f$ such that 
$ \int\Phi(\frac{|f(\omega)|}{\lambda})\,\mathrm{d}\mathbb{P} <\infty$ for all
$0 < \lambda<\infty$
is called the {\em Orlicz heart} $H^\Phi:=H^\Phi(\Omega,\Sigma,\mathbb{P})$.
Obviously, Orlicz spaces  are Banach lattices in the almost everywhere order.
Given an Orlicz function $\Phi$, its {\em conjugate function} $\Psi$ is defined
by 
\[\Psi(t) = \sup\{st -\Phi(s): s \geq 0\} \text{ for $t\geq 0$}.\]
Then $\lim_{t\rightarrow\infty}\frac{\Psi(t)}{t}=\infty$. 
Moreover, $\Psi$ is an Orlicz function
iff
$\lim_{t\rightarrow\infty}\frac{\Phi(t)}{t}=\infty$, and in this case,
the conjugate function of $\Psi$ is $\Phi$. 
Elements of the Orlicz space $L^\Psi$
(in particular, of $H^\Psi$) act on $L^\Phi$ (in particular, on $H^\Phi$), via
the duality
\[ \la f,g\ra = \int fg\,\mathrm{d}\mathbb{P}, \quad g\in L^\Psi,\;\; f \in
L^\Phi.\]
We collect some basic facts concerning Orlicz spaces in the next proposition. 
Refer to \cite[Chapter 2]{ES:02} and \cite{RR:91} for in-depth studies of these
spaces.  

\begin{proposition}\label{orlicz}
Let $\Phi$ be an Orlicz function with conjugate $\Psi$.  
Then $H^\Phi$ is the order continuous part of $L^\Phi$.  In particular, 
$H^\Phi$ is an order continuous Banach lattice.
Moreover,
$L^\Phi \neq L^1$ (as sets or as lattice isomorphic Banach lattices) if and only
if $\lim_{t\to\infty}\frac{\Phi(t)}{t} =\infty$. In this case, $L^\Psi =(L^\Phi)^\sim_n =  (H^\Phi)_n^\sim=(H^\Phi)^*$  
(equality as sets and as lattice isomorphic Banach lattices).
\end{proposition}

Suppose that $X = L^\Phi$, where $\lim_{t\to\infty}\frac{\Phi(t)}{t} = \infty$.
Set $Y = H^\Psi$. 
By Proposition~\ref{orlicz} applied to the conjugate case,  $Y$ is an order continuous Banach lattice with weak units (in fact it contains the constant functions), and $X = Y^*$. 
Thus the equivalence of (1) and (3) in Theorem \ref{general-rep} yields
Theorem~\ref{rep-orl-gx}.

Proposition~\ref{orlicz} identifies $L^\Psi$ as the order continuous duals of
$L^\Phi$ and $H^\Phi$. We now identify $H^\Psi$ are their uo-duals.  
Indeed, by Proposition~\ref{orlicz}, $L^\Psi =(L^\Phi)^\sim_n = 
(H^\Phi)_n^\sim$ and $(L^\Psi)^a = H^\Psi$. Thus by Theorem~\ref{uo-dual-cont},
we obtain the following corollary, where equality again means equality as sets
and with equivalent norms.

\begin{corollary}\label{uo-duals}
Let $\Phi$ be an Orlicz function with conjugate $\Psi$.  Assume that
$\lim_{t\to\infty}\frac{\Phi(t)}{t} = \infty$. Then
$H^\Psi=(L^\Phi)_{uo}^\sim=(H^\Phi)_{uo}^\sim$. 
\end{corollary}

When the Orlicz function $\Phi$ is an N-function and the underlying measure
space is finite and separable, it is proved  in \cite[pp.~148,
Proposition~6]{RR:91} that, for any $g\in H^\Psi$, $\int f_ng\rightarrow
\int fg$ whenever $(f_n) $ is norm bounded in $L^\Phi$ and
$f_n\xrightarrow{a.e.}f$.
In our notation, this means  that $H^\Psi\subset (L^\Phi)_{uo}^\sim$.
Corollary \ref{uo-duals} shows that the technical assumptions in the result can
be removed and also establishes the converse.

Finally, let us consider Proposition~\ref{buo-closure} and Corollary
\ref{uo-closed-all} in the context of Orlicz spaces.
Assume that  $(\Omega,\Sigma,\mathbb{P})$ is a nonatomic probability space and
let $\Phi$ be an Orlicz function so that $\lim_{t\to\infty}\frac{\Phi(t)}{t} =
\infty$. Set $X = H^\Phi$.
 Proposition~\ref{buo-closure} says that $\overline{C}^{buo} =
\overline{C}^{\sigma(H^\Phi,H^\Psi)}$ for every convex set $C$ in $H^\Phi$ if
and only if $L^\Psi=(H^\Phi)^*$ is order continuous, or  equivalently, $\Psi$
satisfies the
$\Delta_2$-condition. 
On the other hand,  Corollary \ref{uo-closed-all} says that every boundedly uo
closed convex set is $\sigma(H^\Phi,H^\Psi)$-closed if and only if
$\sigma(H^\Phi,H^\Psi)$ has the Krein-Smulian property.
In \cite[Theorem~4.3]{GLX:17}, it is shown that the latter holds if and only if
either $\Phi$ or $\Psi$ satisfies the
$\Delta_2$-condition. 
The difference in the two assertions above points to the subtle fact that the
bounded uo closure is not necessarily boundedly uo closed!


\begin{thebibliography}{100}


\bibitem{AB:78} 
C.~Aliprantis, O.~Burkinshaw, \emph{Locally  solid  Riesz  spaces}, Academic
Press, New York, 1978

\bibitem{AB:06}
C.~Aliprantis, O.~Burkinshaw, \emph{Positive Operators}, Springer the
Netherlands,
2006.

\bibitem{BF:10} S.~Biagini, M.~Frittelli, On the extension of the Namioka-Klee
theorem and on the Fatou Property for risk measures. In: \emph{Optimality and
risk-modern
trends in mathematical finance}, Springer Berlin Heidelberg, 2010, 1--28.


\bibitem{B:11}
H.~Brezis, \emph{Functional Analysis, Sobolev Spaces and Partial Differential
Equations}, Springer Science \& Business Media, Berlin, 2010.

\bibitem{ES:02} 
G.A.~Edgar, L.~Sucheston, \emph{Stopping Times and Directed
Processes}, Encyclopedia of Mathematics and Its Applications 47, Cambridge
University Press, Cambridge, 1992.

\bibitem{G:14}
N.~Gao, Unbounded order convergence in dual spaces, \emph{J.\ Math.\ Anal.\ 
Appl.} 419(1), 2014, 347-354.

\bibitem{GLX:17}
N.~Gao, D.~Leung, F.~Xanthos, Dual representation of risk measures on Orlicz
spaces. Preprint: 
{\tt arXiv:1610.08806}


\bibitem{GTX:16}
N.~Gao, V.~Troitsky, F.~Xanthos, Uo-convergence and its applications to Ces\'aro
means in Banach lattices, \emph{Israel J.\ of Math}, to appear. Preprint: {\tt
arXiv:1509.07914}


\bibitem{GX:14}
N.~Gao, F.~Xanthos, Unbounded order convergence and application to martingales
without probability, \emph{J.\ Math.\ Anal.\  Appl.} 415(2), 2014, 931-947.

\bibitem{GX:16a}
N.~Gao, F.~Xanthos, On the C-property and $w^*$-representations of risk
measures, \emph{Mathematical Finance}, to appear. Preprint: {\tt
arXiv:1511.03159 }


\bibitem{HZ:07}
P.~Hajek, V.~Montesinos Santalucia, J.~Vanderwerff, V.~Zizler,
\emph{Biorthogonal Systems in Banach Spaces}, Springer Science \& Business
Media, Berlin,
2007.

\bibitem{MN:91}
P.~Meyer-Nieberg, \emph{Banach lattices}, {Universitext}, {Springer-Verlag},
{Berlin}, {1991}.

\bibitem{RR:91}
M.M.~Rao, Z.D.~Ren, \emph{Theory of Orlicz spaces}, Marcel Dekker, Inc., New
York, 1991.

\bibitem{W:77} 
A.~W.~Wickstead, Weak and unbounded order convergence in {B}anach
lattices, \emph{J.\ Austral.\ Math.\ Soc.\ Ser.\ A} {24}, 1977,
312-319

\end{thebibliography}
\end{document}